\documentclass[12pt]{article}
\usepackage{graphicx} 
\usepackage{amsmath,amssymb,amsthm,amsfonts}
\usepackage{geometry}
\usepackage{amsmath}
\usepackage{xcolor}
\usepackage{mathrsfs}
\usepackage{hyperref}
\hypersetup{colorlinks=true, linkcolor=blue, anchorcolor=blue, citecolor=blue}
\numberwithin{equation}{section}
\numberwithin{figure}{section}
\allowdisplaybreaks

\begin{document}

\newtheorem{theorem}{Theorem}[section]
\newtheorem{definition}{Definition}[section]
\newtheorem{example}{Example}[section]
\newtheorem{corollary}{Corollary}[section]
\newtheorem{lemma}{Lemma}[section]
\newtheorem{proposition}{Proposition}[section]
\newtheorem{remark}{Remark}[section]
\newtheorem{assumption}{Assumption}[section]

\makeatletter
\renewenvironment{proof}[1][\proofname]{\par
  \pushQED{\qed}%
  \normalfont \topsep6\p@\@plus6\p@\relax
  \trivlist
  \item[\hskip\labelsep\bfseries #1\@addpunct{.}]\ignorespaces
}{%
  \popQED\endtrivlist\@endpefalse
}
\makeatother

\title{\bf Large Deviation Principle for Neutral Type Mckean-Vlasov Stochastic Differential Equations}
\author{
{\bf Zhaohang Wang$^{a}$, Junhao Hu$^{b}$, Chenggui Yuan$^{a}$}\\
\footnotesize{$^{a}$Department of Mathematics, Swansea University,
Bay Campus, SA1 8EN, UK}\\
\footnotesize{Email: zhaohang.wang@swansea.ac.uk,\quad c.yuan@swansea.ac.uk}\\
\footnotesize{$^{b}$School of Mathematics and Statistics, South-Central Minzu University, Wuhan 430074, China}\\
\footnotesize{Email: junhaohu74@163.com}\\
}

\date{}

\maketitle

\begin{abstract}
This paper investigates neutral-type McKean-Vlasov stochastic differential equations in which  the drift and diffusion coefficients  depend on both the segment process and its distribution. Under a one-sided Lipschitz condition on the drift coefficient, we establish a Freidlin-Wentzell-type large deviation principle for the solution process by using the extended contraction principle combined with an exponential approximation technique. Our results extend existing large deviation principles for McKean-Vlasov equations to the neutral case.


\end{abstract}

\section{Introduction}
Stochastic differential equations (SDEs) play a fundamental role in modeling complex dynamical systems subject to random perturbations, with applications ranging from physics and biology to economics and engineering. In recent decades, increasing attention has been paid to distribution-dependent stochastic differential equations(DDSDEs), also known as McKean–Vlasov SDEs (MVSDEs), in which the coefficients depend not only on the current state of the process, but also on its distribution. Such equations naturally arise as the mean-field limits of weakly interacting particle systems, and they provide powerful mathematical models for describing collective dynamics in large populations. Classical references on the subject include McKean’s pioneering work on nonlinear Markov processes \cite{MK67}, and subsequent developments by Sznitman \cite{S91} and others. More recently, MVSDEs have found widespread applications in areas such as statistical physics, quantitative finance, control theory, and machine learning \cite{CD18}. For MVSDEs, substantial progress has been made regarding the well-posedness of solutions, the regularity and ergodic properties of the associated transition probabilities, as well as the propagation of chaos(see, e.g., \cite{WFY18, WFY19, MR21, BJH21}). 

Within the broad study of SDEs, an important line of research concerns large deviation principles (LDPs), which provide asymptotic estimates for the probabilities of rare events and quantify the exponential rate at which such probabilities decay. In the study of LDPs for stochastic processes, two main approaches have been developed. The first is the classical method initiated by Freidlin and Wentzell, which relies on an extended contraction principle and exponential approximation argument; see \cite{F70, ZTS06, M84, R19, SYQ20, HG16,WY24}. For example, \cite{F70} fisrt introduced the LDPs for SDEs. In subsequent studies, researchers have also considered stochastic differential equations with delay (also called stochastic functional differential equations), where the future evolution of the system depends not only on its present state but also on its past trajectory. Such equations are of great importance in modeling systems with memory effects, which frequently occur in fields such as population dynamics, engineering, and finance. For instance, \cite{ZTS06} and \cite{M84} developed the LDPs for SDEs with constant delay. \cite{SYQ20} established the LDPs for neutral stochastic functional differential equations(SFDEs) with finite delay. More recently, \cite{WY24} established the LDPs for regime-switching SDEs with infinite delay. Moreover, see \cite{R19} for the LDPs for MVSDEs.

Besides the classical approach, another powerful framework has been developed in recent years, namely the weak convergence method introduced in \cite{Dup97}. Based on a variational representation of positive functionals of Brownian motion introduced in \cite{Dup00}, this method allows the LDPs to be reformulated as a stochastic control problem. By establishing weak convergence of controlled processes, one can then derive the desired LDPs, see\cite{Dup08, Dup11, Dup16, Dup19}. For instance, \cite{ML13} and \cite{CA14} introduced the LDPs for SDEs with constant delay.
\cite{BJH15} established the LDPs for a class of neutral SFDEs driven with jumps.  
also see \cite{HL21, AR22, LS22, SYL23} for the LDPs for MVSDEs.

Although significant progress has been made in applying the extended contraction principle to establish LDPs for various classes of stochastic equations, certain important settings remain largely unexplored. For instance, \cite{SYQ20} investigated the LDP for neutral SFDEs under a one-sided Lipschitz condition on the drift, whereas \cite{R19} proved that the solutions of MVSDEs satisfy the LDP in the uniform topology on the path space. Both works rely on the extended contraction principle, which has proved to be a powerful tool in this context. However, to the best of our knowledge, there are no results available on the LDP for neutral MVSDEs obtained via the extended contraction principle. The purpose of this paper is to fill this gap by establishing sample-path LDPs for solutions to neutral MVSDEs (\ref{1.1}) under the assumption (\ref{A1.1}) stated below. Our results thus extend the scope of the extended contraction principle to a new class of distribution-dependent non-Markovian systems, combining neutral-type dynamics with mean-field interactions.

To establish the LDP for the class of MVSDEs considered in this paper, we first introduce the basic notations will be used throughout the paper. For any positive integer $d$, let $(\mathbb{R}^d, \langle\cdot,\cdot\rangle, |\cdot|)$ be the Euclidean space of $d$ dimensions with inner product $\langle\cdot,\cdot\rangle$ and induced norm $|\cdot|$. Let $\mathbb{M}^{d\times m}$ be the collection of all $d\times m$ matrices with the Hilbert-Schimidt norm $\|\sigma\|_{HS}=\sqrt{{\rm tr}(\sigma^*\sigma)}$, where $\sigma^*$ is the transpose of the matrix $\sigma$. For any fixed real number $\tau>0$, denote by $\mathscr{C}=C([-\tau,0];\mathbb{R}^d)$ the space of all continuous functions from $[-\tau,0]$ to $\mathbb{R}^d$, which is a Polish space under the uniform norm $\|f\|_\infty=\sup_{t\in[-\tau,0]}|f(t)|$ for any $f\in\mathscr{C}$. Denote $H([0,T];\mathbb{R}^m):=\{\varphi(s):[0,T]\to\mathbb{R}^m;\, \varphi \mbox{ is absolutely continuous with } \varphi(0)=0 \mbox{ and } \int_0^T|\dot{\varphi}(s)|^2\,ds<\infty\Big{\}}$, which is a Banach space equipped with the norm $ \|f\|_H=(\int_0^T| \dot{f}(s)|^2\, ds)^\frac{1}{2}.$ Let $\mathscr{P}_2(\mathscr{C})$ be a subset  of all probability measures $\mathscr{P}(\mathscr{C})$ on $\mathscr{C}$ with finite second moments, i.e.
$$
\mathscr{P}_2(\mathscr{C})=\Big{\{}\mu\in\mathscr{P}(\mathscr{C}):\mu(\|\cdot\|_\infty^2)=\int_\mathscr{C}\|\xi\|_\infty^2\,\mu(d\xi)<\infty\Big{\}},
$$
which is a Polish space under the Wasserstein distance
$$
\mathbb{W}_2(\mu,\nu)=\inf_{\pi\in\Gamma(\mu,\nu)}\Big(\int_{\mathscr{C}\times\mathscr{C}}\|\xi-\eta\|_\infty^2\,\pi(d\xi,d\eta)\Big)^\frac{1}{2},
$$
where $\Gamma(\mu,\nu)$ is the set of all couplings of probability measures $\mu$ and $\nu$.
For any fixed time horizon $T>0$, we consider the following MVSDE
\begin{equation}\label{1.1}
\begin{cases}
    d(X^\epsilon(t)-D(X^\epsilon_t))=b(X^\epsilon_t, \mathscr{L}_{X^\epsilon_t})\,dt+\sqrt{\epsilon}\sigma(X^\epsilon_t, \mathscr{L}_{X^\epsilon_t})\,dW(t),\quad t\in[0,T], \\
       X^\epsilon_0=\xi\in\mathscr{C}.
\end{cases}
\end{equation}
where $\{X^\epsilon_t\}_{t\in[0,T]}$ is the corresponding segment process of $\{X^\epsilon(t)\}_{t\in[0,T]}$, i.e.
$$
X^\epsilon_t(
\theta)=X^\epsilon(t+\theta),\quad \theta\in[-\tau,0],
$$
which is a continuous function in $\mathscr{C}$ for any $t\in[0,T]$ and $\mathscr{L}_{X^\epsilon_t}$ is the law of $X^\epsilon_t$. The coefficients $D:\mathscr{C}\to\mathbb{R}^d$, $b:\mathscr{C}\times\mathscr{P}_2(\mathscr{C})\to\mathbb{R}^d$ and $\sigma:\mathscr{C}\times\mathscr{P}_2(\mathscr{C})\to\mathbb{M}^{d\times m}$ are measurable functions.  $\epsilon\in[0,1]$ is a small parameter that controls the strength of the stochastic perturbation, and the asymptotic behavior as $\epsilon\to 0$ will be of particular interest and $\{W(t)\}_{t\ge0}$ is a $m$-dimensional Brownian motion  defined in a filtered complete probability space $(\Omega, \mathscr{F},\{\mathscr{F}_t\}_{t\ge0},\mathbb{P})$.

It is hard to directly apply the extended contraction principle Lemma \ref{L2.1} to prove that the law of the solution to system (\ref{1.1}) satisfies the LDP on $C([-\tau,T];\mathbb{R}^d)$ because the coefficients depend on its own distribution. To overcome this difficulty, our methodology is to first construct an approximation of $\mathscr{L}_{X^\epsilon_t}$. For this purpose, we introduce the following ordinary differential equation (ODE): 
\begin{equation}
      d(X^0(t)-D(X^0_t))=b(X^0_t, \delta_{X^0_t})dt,\quad X^0_0=\xi, \quad t\in[0,T],
\end{equation}
where $\delta_{X^0_t}$ is the Dirac measure of $X^0_t$. 
In the small noise limit $\varepsilon \to 0$ for the system (\ref{1.1}), the diffusion term vanishes, and the equation is reduced to an ordinary differential equation of the above form. Under Assumptions (\ref{A1.1}), the above ODE admits a unique solution. Note that $\mathscr{L}_{X^\epsilon_t}$ converges in distribution to $\delta_{X^0_t}$ as $\varepsilon \to 0$. Motivated by this observation, we replace  $\mathscr{L}_{X^\epsilon_t}$ in equation (\ref{1.1}) by $\delta_{X^0_t}$, which leads to the following approximate SDE.
\begin{equation}
\begin{cases}
            d(Y^\epsilon(t)-D(Y^\epsilon_t))=b(Y^\epsilon_t, \delta_{X^0_t})\,dt+\sqrt{\epsilon}\sigma(Y^\epsilon_t, \delta_{X^0_t})\,dW(t),\quad t\in[0,T],\\
            Y^\epsilon_0=\xi.
\end{cases}
\end{equation}
 Our approach to prove the LDP of the law of $X^\varepsilon(\cdot)$ proceeds in two main steps.  The first step is to establish an LDP for the law of $Y^{\varepsilon}(\cdot)$ by applying the approximation techniques developed later, see Theorem \ref{T2.1}(i). Next, we show that $X^\varepsilon(t)$ and $Y^\varepsilon(t)$ are exponentially equivalent. Therefore, by the stability of the LDP under exponential equivalence, see \cite{DZ}, it follows that the law of $X^\varepsilon(\cdot)$ satisfies the same LDP, see Theorem \ref{T2.1}(ii).

The Assumptions on the drift and diffusion coefficients are as follows.
\begin{assumption}\label{A1.1}
     For any $\xi, \eta \in \mathscr{C}$ and $\mu, \nu \in \mathscr{P}_2(\mathscr{C})$, we impose the following conditions.
\begin{itemize}
  \item[(A1)] There is a constant $\alpha \in (0, 1)$ such that
  \[
  |D(\xi) - D(\eta)| \leq \alpha \|\xi - \eta\|_\infty, \quad D(0) = 0.
  \]

  \item[(A2)] The functions $b$ and $\sigma$ are bounded on bounded subsets of $\mathscr{C} \times \mathscr{P}_2(\mathscr{C})$. Moreover, there exists a constant $L > 0$ such that
  \[
  2 \langle \xi(0) - \eta(0) - (D(\xi) - D(\eta)), b(\xi, \mu) - b(\eta, \nu) \rangle 
  \leq L \left(\|\xi - \eta\|_\infty^2 +\mathbb{W}_2(\mu, \nu)^2 \right),
  \]
  \[
  \|\sigma(\xi, \mu) - \sigma(\eta, \nu)\|_{\mathrm{HS}}^2 
  \leq L \left( \|\xi - \eta\|_\infty^2 + \mathbb{W}_2(\mu, \nu)^2 \right).
  \]
  \item [(A3)] There exists a constant $L_1>0$ such that
  \[
  |b(0,\mu)|^2\vee\|\sigma(0,\mu)\|_{HS}^2\le L_1(1+\mu(\|\cdot\|_\infty^2)).
  \]
\end{itemize}
\end{assumption}

\begin{remark}
    From (A1)-(A3), it follows that for any $\xi \in \mathscr{C}$ and $\mu \in \mathscr{P}_2(\mathscr{C})$
\begin{equation}\label{1.2}
2\langle \xi(0) - D(\xi),\, b(\xi,\mu) \rangle \vee\|\sigma(\xi, \mu) \|_{\mathrm{HS}}^2 \leq L_2(1 + \|\xi\|^2_{\infty}+\mu(\|\cdot\|_\infty^2)),  
 \quad|D(\xi)| \leq \alpha \|\xi\|_{\infty}, 
\end{equation}
 where $L_2=\max{\{L+(1+\alpha)^2, L+L_1\}}$.
\end{remark}
The remainder of the paper is organized as follows. In Section 2, we present some preliminary lemmas and state the main results of the paper. Section 3 is devoted to the proof of Theorem \ref{T2.1} (i). We begin by applying Lemma \ref{L2.1} to show that the law of $Y^\epsilon(\cdot)$ satisfies the LDP under Assumption \ref{A1.1}, together with the additional boundedness of $b$ and $\sigma$. Subsequently, by removing the boundedness condition, we complete the proof of Theorem \ref{T2.1}(i) by using the definition of LDP. Finally, in Section 4, we establish Theorem \ref{T2.1}(ii), which constitutes the main result of this work.

\section{Preliminaries and Main Results}
We begin with the definition of LDP.
\begin{definition}[Large Deviation Principle] 
    Let $\{\mu_\varepsilon\}$ be a family of probability measures defined on a Polish space $\mathscr{X}$. If there exists a lower semicontinuous function $I:\mathscr{X}\to[0,\infty]$ (i.e. the level set $I^{-1}([0,c])$ is a closed subset of $\mathscr{X}$ for any $c\in[0,\infty)$) such that 
    \begin{itemize}
    \item[(i)] For any closed subset \( C \subset \mathscr{X} \),
    \begin{equation}
    \limsup_{\varepsilon \to 0} \varepsilon \log \mu_\varepsilon(C) \le - \inf_{x \in C} I(x).
    \end{equation}
    
    \item[(ii)] For any open subset \( G \subset \mathscr{X} \),
    \begin{equation}
    \liminf_{\varepsilon \to 0} \varepsilon \log \mu_\varepsilon(G) \ge - \inf_{x \in G}I(x).
    \end{equation}
\end{itemize}
Then $I$ is called a rate function , and $\{\mu_\varepsilon\}$ is said to satisfy the LDP with rate function $I$. Moreover, if all level sets are compact sets in $\mathscr{X}$, then $I$ is called a good rate function. 
\end{definition}
According to the Schilder theorem \cite{DZ}, the law of Brownian motion $\sqrt{\epsilon}W(\cdot)$ on $C([0, T]; \mathbb{R}^m)$, denoted by $\nu_\epsilon$, satisfies the LDP with the rate function $I(\varphi)$ defined as follows.
\begin{equation}\label{2.00}
    I(\varphi)= \left\{ \begin{array}{rcl}
\frac{1}{2}\int_0^T|\dot{\varphi}(s)|^2\,ds, 
& if \quad \varphi\in H([0,T];\mathbb{R}^m); \\ 
+\infty,  & otherwise.
\end{array}\right.
\end{equation}
To establish the LDP for the law of the solution to system (\ref{1.1}), we begin by stating a series of lemmas that form the foundation for our main results. Firstly, we recall an extended contraction principle in \cite{DZ}, which establishes that the LDP satisfied by one sequence of probability measures can be transferred to another via a sequence of continuous mappings.

\begin{lemma}\label{L2.1}
Let $\mathbb{X}$ be a Hausdorff topological space  and $(\mathbb{Y}, d)$ be a metric space.  Let $\{\nu_\varepsilon\}$ be a family of probability measures that satisfies the LDP with a good rate function $I$  in $\mathbb{X}$, and for $n = 1, 2, \dots$, let $h_n : \mathbb{X} \to \mathbb{Y}$ be continuous functions in $(\mathbb{Y}, d)$. Assume there exists a measurable map $h: \mathbb{X} \to \mathbb{Y}$ such that for every  $r < \infty$,
\begin{equation}
\label{eq:2.2}
\limsup_{n \to \infty} \sup_{\{x : I(x) \leq r\}} d(h_n(x), h(x)) = 0.
\end{equation}
Then any family of probability measures $\{\widetilde{\nu}_\varepsilon\}$ for which $\{\nu_\varepsilon \circ h_n^{-1}\}$ are exponentially good approximations satisfies the LDP in $\mathbb{Y}$ with the good rate function
\[
\widetilde{I}(y) = \inf\{ I(x) \;|\; y = h(x) \}.
\]
\end{lemma}

To control the probability of large deviations of an It\^o process with bounded coefficients over a finite interval, we employ the following exponential estimate. For more details, see \cite{S}.
\begin{lemma}\label{L2.2}
Let $\alpha : [0, \infty) \times \Omega \to \mathbb{M}^{d\times m}$ and $\beta : [0, \infty) \times \Omega \to \mathbb{R}^d$ be $(\mathcal{F}_t)_{t \geq 0}$-progressively measurable processes. Assume that $\|\alpha(\cdot)\|_{\text{HS}} \leq A$ and $|\beta(\cdot)| \leq B$. Define
\[
\xi(t) := \int_0^t \alpha(s)\,dW(s) + \int_0^t \beta(s)\,ds, \quad \text{for } t \geq 0.
\]
Let $T > 0$ and $R > 0$ satisfy $\sqrt{d}\,BT < R$. Then
\begin{equation}
\label{eq:2.3}
\mathbb{P} \left( \sup_{0 \leq t \leq T} |\xi(t)| \geq R \right)
\leq 2d \exp\left( -\frac{(R - \sqrt{d}\,BT)^2}{2A^2 d T} \right).
\end{equation}
\end{lemma}

Following the approach in \cite{HY21}, it can be shown that system (\ref{1.1}) admits a unique solution $\{X^\epsilon(t)\}_{t\in[-\tau,T]}$ under assumptions (A1)-(A3).
Moreover, we can derive the following uniform moment estimate for the process $\{X^\epsilon(t)\}_{t\in[-\tau,T]}$
\begin{lemma}
     Under assumptions (A1)-(A3), it holds that
    \begin{align}\label{1.3}
      & \mathbb{E}\big[\sup_{0\le t \le T}\|X^\epsilon_t\|_\infty^2\big]\nonumber\\
       \le&\frac{(2\alpha^2+3\alpha+3)\|\xi\|_\infty^2+2L_2T(1+65\epsilon)}{(1-\alpha)^2}\exp\Big\{\frac{4L_2T(1+65\epsilon)}{(1-\alpha)^2}\Big\}=:L_3(\epsilon).
    \end{align}
\end{lemma}
\begin{proof}
    Denoting $\tilde{Z}(t)=X^\epsilon(t)-D(X^\epsilon_t)$ and applying Itô's formula to $|\tilde{Z}(t)|^2$, we have
    \begin{align*}
        |\tilde{Z}(t)|^2=&|\tilde{Z}(0)|^2+\int_0^t \Big(2 
\langle \tilde{Z}(s), b(X_s^\varepsilon,\mathscr{L}_{X^\epsilon_s})  \rangle + \varepsilon  \| \sigma(X_s^\varepsilon,\mathscr{L}_{X^\epsilon_s}) \|_{\mathrm{HS}}^2\Big) \, ds \\
+&\sqrt{\epsilon}\int_0^t 2 
\langle \tilde{Z}(s), \sigma(X_s^\varepsilon,\mathscr{L}_{X^\epsilon_s})dW(s)  \rangle  .
    \end{align*}
    By the  Burkholder-Davis-Gundy (BDG) inequality and (\ref{1.2}), one has
    \begin{equation*}
       \mathbb{E}[\sup_{0\le t\le T}|\tilde{Z}(t)|^2]\le2(1+\alpha)^2\|\xi\|_\infty^2+2L_2T(1+65\epsilon)+4L_2(1+65\epsilon)\int_0^T \mathbb{E}[\sup_{0\le u\le s}\|X^\epsilon_u\|_\infty^2]\,ds.
    \end{equation*}
 By the elementary inequality $(a+b)^2\le(1+\kappa)(a^2+\frac{1}{\kappa}b^2)$ for any $\kappa>0$, let $\kappa=\frac{\alpha}{1-\alpha}$, one can get
\begin{align}\label{1.4}
    \sup_{0\le t \le T}\|X^{\epsilon}_t\|_\infty^2&\le\|\xi\|_\infty^2+\sup_{0\le t\le T}|D(X^\epsilon_t)+\tilde{Z}(t)|^2\nonumber\\
    &\le\|\xi\|_\infty^2+\alpha\sup_{0\le t \le T}\|X^{\epsilon}_t\|_\infty^2+\frac{1}{1-\alpha}\sup_{0\le t\le T}|\tilde{Z}(t)|^2,  
\end{align}
therefore,
\begin{align*}
    \mathbb{E}[\sup_{0\le t \le T}\|X^{\epsilon}_t\|_\infty^2]&\le\frac{1}{1-\alpha}\|\xi\|_\infty^2+\frac{1}{(1-\alpha)^2}\mathbb{E}[\sup_{0\le t\le T}|\tilde{Z}(t)|^2]\\
    &\le \frac{2\alpha^2+3\alpha+3}{(1-\alpha)^2}\|\xi\|_\infty^2+\frac{2L_2T(1+65\epsilon)}{(1-\alpha)^2}\\
    &+\frac{4L_2(1+65\epsilon)}{(1-\alpha)^2}\int_0^T \mathbb{E}[\sup_{0\le u\le s}\|X^\epsilon_u\|_\infty^2]\,ds.
\end{align*}
Hence, (\ref{1.3}) follows from Gronwall's inequality.
\end{proof}

 Moreover, we have the uniform bound estimates of $X^0_t$ and $X^\epsilon_t$.
\begin{lemma}
   Under Assumptions (A1)-(A3), it holds that
    \begin{equation}\label{2.1}
        \sup_{0\le t\le T}\|X^0_t\|_\infty^2\le \frac{(\alpha^2+\alpha+2)\|\xi\|_\infty^2+L_2T}{(1-\alpha)^2}e^{\frac{2L_2T}{(1-\alpha)^2}},
    \end{equation}
    and
        \begin{equation}\label{2.2}
       \mathbb{E}\big[\sup_{0\le t \le T}\|X^\epsilon_t-X^0_t\|_\infty^2\big]
       \le\frac{130\epsilon L_2T(1+2L_3(\epsilon))}{(1-\alpha)^2}\exp\Big\{\frac{4L}{(1-\alpha)^2}\Big\}=:L_4(\epsilon).
    \end{equation}
\end{lemma}
\begin{proof}
Denote $Y^0(t)=X^0(t)-D(X^0_t)$. Utilizing (\ref{1.2}) leads to 
\begin{equation*}
|Y^0(t)|^2\le(1+\alpha)^2\|\xi\|_\infty^2+L_2\int_0^t(1+2\|X^0_s\|_\infty^2)\,ds
\end{equation*}
By an argument similar to (\ref{1.4}), it follows that
\begin{align*}
    \sup_{0\le t \le T}\|X^0_t\|_\infty^2&\le\frac{1}{1-\alpha}\|\xi\|_\infty^2+\frac{1}{(1-\alpha)^2}\sup_{0\le t\le T}|Y^0(t)|^2\\
    &\le \frac{(\alpha^2+\alpha+2)\|\xi\|_\infty^2+L_2T}{(1-\alpha)^2}
    +\frac{2L_2}{(1-\alpha)^2}\int_0^T \sup_{0\le u\le s}\|X^0_u\|_\infty^2\,ds.
\end{align*}
Hence, (\ref{2.1}) follows from Gronwall's inequality. Denote $\bar{Z}^{\epsilon}(t)=X^{\epsilon}(t)-X^{0}(t)-(D(X^\epsilon_t)-D(X^{0}_t))$. Applying Itô's formula, the BDG inequality and condition (A2) yields
\begin{align*}
           \mathbb{E}[\sup_{0\le t\le T}|\bar{Z}(t)|^2]&\le130\epsilon L_2T+4L\int_0^T \mathbb{E}[\sup_{0\le u\le s}\|X^\epsilon_u-X^0_u\|_\infty^2]\,ds\\
           &+260\epsilon L_2\int_0^T \mathbb{E}[\sup_{0\le u\le s}\|X^\epsilon_u\|_\infty^2]\,ds.
\end{align*}
A simple computation shows that 
$$
\sup_{0\le t\le T}\|X^\epsilon_t-X^0_t\|_\infty\le \frac{1}{1-\alpha}\sup_{0\le t\le T}|\bar{Z}(t)|,
$$
therefore,
\begin{equation*}
    \mathbb{E}\big[\sup_{0\le t \le T}\|X^\epsilon_t-X^0_t\|_\infty^2\big]
       \le\frac{130\epsilon L_2T(1+2L_3(\epsilon))}{(1-\alpha)^2}+\frac{4L}{(1-\alpha)^2}\int_0^T \mathbb{E}[\sup_{0\le u\le s}\|X^\epsilon_u-X^0_u\|_\infty^2]\,ds.
\end{equation*}
Hence, (\ref{2.2}) follows from Gronwall's inequality.
\end{proof}
 Define a map $M:H([0,T];\mathbb{R}^m)\to C([-\tau,T];\mathbb{R}^d)$ that for any $\varphi\in H([0,T];\mathbb{R}^m)$, $M(\varphi)$ as the solution of the following ODE:
\begin{align}
        M(\varphi)(t)-D(M_t(\varphi))&=M(\varphi)(0)-D(M_0(\varphi))\nonumber\\
        &+\int_0^tb(M_s(\varphi), \delta_{X^0_s})ds+\int_0^t\sigma(M_s(\varphi), \delta_{X^0_s})\dot{\varphi}(s)ds,
\end{align}
where $t\in[0,T]$ and $M_t(\varphi)(\theta)=M(\varphi)(t+\theta)$, $\theta\in[-\tau,0]$ and $M_0(\varphi)=\xi$. Let $I(\cdot)$ be defined by (\ref{2.00}) and define
\[
\widehat{I}(f) := \inf \left\{ I(\varphi) \;\middle|\; M(\varphi) = f,\; \varphi \in H([0,T];\mathbb{R}^m) \right\}, \quad f \in C([-\tau, T]; \mathbb{R}^d).
\]
 We now state the main results.

\begin{theorem}\label{T2.1}
    Assume (A1)-(A3) hold. Then:
\begin{itemize}
 \item[{\rm (i)}] The law of $Y^{\varepsilon}(\cdot)$ on \(C([-\tau, T]; \mathbb{R}^d)\) satisfies the LDP with  rate function $\widehat{I}(f)$;

\item[{\rm (ii)}]  The law of \(X^\varepsilon(\cdot)\) on \(C([-\tau, T]; \mathbb{R}^d)\) satisfies the LDP with the same rate function $\widehat{I}(f)$.
\end{itemize}
\end{theorem}

We construct the following process \( Y^{\varepsilon,n}(\cdot) \) by using an approximate scheme. For a positive real number \( t \), let $\lfloor t \rfloor$ be its integer part. Given any positive integer \( n \ge 1 \), we consider the following equation 
\begin{equation}
    \begin{cases}
    \mathrm{d}(Y^{\varepsilon,n}(t) - D\left(Y^{\varepsilon,n}_t\right)) = b\left(Y^{\varepsilon,n}_t,\delta_{X^0_t}\right)\,\mathrm{d}t + \sqrt{\varepsilon}\,\sigma\left(\bar{Y}^{\varepsilon,n}_t,\delta_{\bar{X}^0_t}\right)\,\mathrm{d}W(t), \quad t \in[0,T],\\
     Y^{\varepsilon,n}_0 = \xi,
\end{cases}
\end{equation}
where
\[
\bar{Y}^{\varepsilon,n}_t(\theta) := Y^{\varepsilon,n}((t + \theta) \wedge t_n),\quad \bar{X}^{0}_t(\theta) := X^{0}((t + \theta) \wedge t_n),\quad t_n := \frac{\lfloor nt \rfloor}{n}, \quad \ \theta \in [-\tau, 0].
\]
\begin{remark}
    Based on assumption (A1), we can give some estimates. The Lipschitz continuity of the neutral term $D$ allows us to control the difference between two segment processes in terms of their difference after removing the neutral term. Notice that
\begin{align*}
    &\sup_{0 \leq t \leq T}\| Y_t^{\varepsilon}-Y_t^{\varepsilon,n} \|_\infty
    =\sup_{0 \leq t \leq T}| Y^{\varepsilon}(t)-Y^{\varepsilon,n}(t)|\nonumber\\
    \leq &\sup_{0 \leq t \leq T}|Y^{\varepsilon}(t)-Y^{\varepsilon,n}(t)-(D(Y^\epsilon_t)-D(Y^{\epsilon,n}_t))|+\alpha\sup_{0 \leq t \leq T}\| Y_t^{\varepsilon}-Y_t^{\varepsilon,n} \|_\infty,
\end{align*}
then we have
\begin{equation}\label{2.3}
    \sup_{0 \leq t \leq T}\| Y_t^{\varepsilon}-Y_t^{\varepsilon,n} \|_\infty\le\frac{1}{1-\alpha}\sup_{0 \leq t \leq T}|Y^{\varepsilon}(t)-Y^{\varepsilon,n}(t)-(D(Y^\epsilon_t)-D(Y^{\epsilon,n}_t))|.
\end{equation}
In reverse, it is easy to see that for any $t\in[0,T]$
\begin{equation}\label{2.4}
   |Y^{\varepsilon}(t)-Y^{\varepsilon,n}(t)-(D(Y^\epsilon_t)-D(Y^{\epsilon,n}_t))|\le (1+\alpha)  \| Y_t^{\varepsilon}-Y_t^{\varepsilon,n} \|_\infty.
\end{equation}
These observations will be used later when we estimate the approximation error.
\end{remark}
Define the map
$
M^n(\cdot): C([0, T], \mathbb{R}^m) \to C([-\tau, T], \mathbb{R}^d)
$
by
\begin{align}
M^n(\varphi)(t) - D\left(M^n_t(\varphi)\right) & =
M^n(\varphi)(t_n) - D\left(M^n_{t_n}(\varphi)\right)
+ \int_{t_n}^{t} b\left(M^n_s(\varphi),\delta_{X^0_s}\right)\,\mathrm{d}s\nonumber\\
&+ \sigma\left(\bar{M}^{n}_t(\varphi),\delta_{\bar{X}^0_{t}}\right)\left[\varphi(t) - \varphi(t_n),\right], \quad t_n \leq t \leq t_n + \tfrac{1}{n},    
\end{align}
with initial condition $M^n_0(\varphi)(t) = \xi(t)$, $-\tau \leq t \leq 0$. Here, for any \( t \in [0,T] \), define $M^n_t(\varphi)(\theta) := M^n(\varphi)(t + \theta)$ and $ \bar{M}^{n}_t(\varphi)(\theta) := M^n(\varphi)\big((t + \theta) \wedge t_n\big)$ for $\theta \in [-\tau, 0]$.
Therefore, $Y^{\varepsilon,n}(t)=M^n(\sqrt{\epsilon}W)(t)$ is a continuous map. 

\section{Proof of (i) of Theorem 2.1 }
In this section, we first consider the case where the drift terms $b$
and $\sigma$ are bounded, that is, satisfying the condition (\ref{3.1}). By combining Lemmas \ref{L2.1} and \ref{L2.2}, we show that under assumptions (A1)-(A3) and (\ref{3.1}), the process \(Y^\varepsilon(t)\) satisfies the LDP (see Lemma \ref{L3.3}). To extend this result beyond the bounded setting, we  remove the boundedness assumption on $b$ and $\sigma$, and define the bounded drift and diffusion coefficients $b_R$ and $\sigma_R$ using the truncation mapping $\chi_R$. By applying the definition of the LDP, we further prove that under assumptions (A1)-(A3), the process \(Y^\varepsilon(t)\) still satisfies the LDP (see Theorem \ref{T2.1}(i)).
\subsection{Case of Bounded Coefficients}
The next lemma provides a uniform exponential estimate of the error between the continuous time process \(Y^\varepsilon(t)\) and its time-discretized counterpart \(Y^{\varepsilon,n}(t)\). 
\begin{lemma}\label{L3.1}
    Let assumptions (A1) and (A2) hold. Moreover, suppose that the coefficients $b$ and $\sigma$ of system (\ref{1.1}) are bounded, i.e. there exists a positive constant $L_5$ such that 
    \begin{equation}\label{3.1}
        |b(\xi, \mu)|\vee\|\sigma(\xi, \mu)\|_{HS}\le L_5,\, \forall \xi\in\mathscr{C}, \forall\mu\in\mathscr{P}_2(\mathscr{C}).
    \end{equation}
    Then, for any $\delta>0$, it holds that
    \begin{equation}
        \lim_{n\to\infty}\limsup_{\epsilon\to0}\epsilon\log\mathbb{P}(\sup_{-\tau\le t\le T}|Y^\epsilon(t)-Y^{\epsilon,n}(t)
        |>\delta)=-\infty . 
\end{equation}
\end{lemma}
\begin{proof}
    Let $Z^{\epsilon,n}(t)=Y^{\epsilon}(t)-Y^{\epsilon,n}(t)$, for any $\kappa>0$ define $\gamma^{\epsilon,n}(\kappa)=\inf\{t\ge0: \|Y^{\epsilon,n}_t-\bar{Y}^{\epsilon,n}_t \|_\infty\ge\kappa\}$, $Z^{\epsilon,n}_\kappa(t)=Z^{\epsilon,n}(t\wedge\gamma^{\epsilon,n}(\kappa))$, and $\widetilde{\gamma}^{\epsilon,n}(\delta)=\inf\{t\ge0: |Z^{\epsilon,n}_\kappa(t)|\ge\delta\}$. With these stopping times, the probability of interest can be decomposed as
    \begin{align}
        &\mathbb{P}(\sup_{0\le t\le T}|Z^{\epsilon,n}(t)
        |>\delta)\nonumber\\
        =&\mathbb{P}(\sup_{0\le t\le T}|Z^{\epsilon,n}(t)
        |>\delta, \gamma^{\epsilon,n}(\kappa)\le T)+\mathbb{P}(\sup_{0\le t\le T}|Z^{\epsilon,n}(t)
        |>\delta, \gamma^{\epsilon,n}(\kappa)> T)\nonumber\\
        \le &\mathbb{P}(\gamma^{\epsilon,n}(\kappa)\le T)+\mathbb{P}(\widetilde{\gamma}^{\epsilon,n}(\delta)\le T).
    \end{align}
We first estimate $\mathbb{P}(\gamma^{\epsilon,n}(\kappa) \le T)$.  
For a fixed $t\in[0,T]$ and any $\theta\in[-\tau,0]$, we expand
\begin{align*}
&|Y_{t}^{\varepsilon,n}(\theta) - \bar{Y}_{t}^{\varepsilon,n}(\theta)| \\
=& |Y^{\varepsilon,n}(t + \theta) - Y^{\varepsilon,n}((t + \theta) \wedge t_n)| \\
=&|\big(Y^{\varepsilon,n}(t + \theta) - Y^{\varepsilon,n}(t + \theta)\big)\mathbf{1}_{\{t+\theta \leq t_n\}} 
+ \big(Y^{\varepsilon,n}(t + \theta) - Y^{\varepsilon,n}(t_n)\big)\mathbf{1}_{\{t+\theta> t_n\}} |\\
=&|\big(Y^{\varepsilon,n}(t + \theta) - Y^{\varepsilon,n}(t_n)\big)\mathbf{1}_{\{t+\theta> t_n\}}| .
\end{align*}
Then, by assumption (A1)
\begin{align*}
&\| Y_{t}^{\varepsilon,n} - \bar{Y}_{t}^{\varepsilon,n} \|_{\infty} \\
\leq &\sup_{t_n - t \leq \theta \leq 0}\Big|D(Y_{t+\theta}^{\varepsilon,n}) - D(Y_{t_n}^{\varepsilon,n}) 
+ \int_{t_n}^{t+\theta} b(Y_s^{\varepsilon,n},\delta_{X^0_s}) \, ds 
+ \sqrt{\varepsilon}\int_{t_n}^{t+\theta}  \, \sigma(\bar{Y}_s^{\varepsilon,n},\delta_{\bar{X}^0_s}) \, dW(s)\Big|\\
\leq &\sup_{t_n - t \leq \theta \leq 0}\alpha\Big\|Y_{t+\theta}^{\varepsilon,n} - Y_{t_n}^{\varepsilon,n}\Big\|_\infty 
+\Big| \int_{t_n}^{t+\theta} b(Y_s^{\varepsilon,n},\delta_{X^0_s}) \, ds 
+ \sqrt{\varepsilon}\int_{t_n}^{t+\theta}  \, \sigma(\bar{Y}_s^{\varepsilon,n},\delta_{\bar{X}^0_s}) \, dW(s)\Big|.
\end{align*}
This implies
\begin{align} \label{3.5}
&\sup_{0 \leq t \leq T} \| Y_{t}^{\varepsilon,n} - \bar{Y}_{t}^{\varepsilon,n} \|_{\infty}\nonumber\\ 
\leq &\frac{1}{1 - \alpha} 
\sup_{0 \leq t \leq T} \sup_{t_n - t \leq \theta \leq 0}
\left| \int_{t_n}^{t+\theta} b(Y_s^{\varepsilon,n},\delta_{X^0_s}) \, ds 
+  \sqrt{\varepsilon}\int_{t_n}^{t+\theta} \, \sigma(\bar{Y}_s^{\varepsilon,n},\delta_{\bar{X}^0_s})\, dW(s)
\right|.
\end{align}
Taking (\ref{3.1}) into consideration and utilizing Lemma 2.2, one gets that
\begin{align*}
&\mathbb{P}(\gamma^{\epsilon,n}(\kappa)\le T)\\
\leq &\mathbb{P} \left( 
\sup_{0 \leq t \leq T} 
\| Y_{t}^{\varepsilon,n} - \bar{Y}_{t}^{\varepsilon,n} \|_{\infty} \geq \kappa \right) \\
\leq &2d \exp \left( 
- \frac{(n \kappa (1 - \alpha) - \sqrt{d} L_5)^2}{2n L_5^2  d \varepsilon} 
\right),
\end{align*}
provided that $ \frac{\sqrt{d} L_5}{n(1-\alpha)}  < \kappa$. Taking limits, we conclude
\begin{equation} \label{3.9}
\lim_{n \to \infty} \limsup_{\varepsilon \to 0} \, \varepsilon \log \mathbb{P}(\gamma^{\epsilon,n}(\kappa) \leq T) = -\infty.
\end{equation}
It remains to estimate the probability $\mathbb{P}(\widetilde{\gamma}^{\epsilon,n}(\delta)\le T)$. We now introduce $G^{\epsilon.n}(t)=Y^{\epsilon}(t)-Y^{\epsilon,n}(t)-(D(Y^\epsilon_t)-D(Y^{\epsilon,n}_t))$. For $\lambda>0$,  
Applying It\^o's formula to $(\kappa^2+|G^{\epsilon.n}(t)|^2)^\lambda$ yields
\begin{align}  \label{3.2} 
&(\kappa^2+|G^{\epsilon.n}(t)|^2)^\lambda \nonumber\\
=& \kappa^{2\lambda}
+ 2\lambda\int_0^t  (\kappa^2 + |G^{\varepsilon,n}(s)|^2)^{\lambda - 1} 
\left\langle G^{\varepsilon,n}(s), b(Y_s^\varepsilon,\delta_{X^0_s}) - b(Y_s^{\varepsilon,n},\delta_{X^0_s}) \right\rangle \, ds\nonumber\\
+& \lambda \varepsilon\int_0^t(\kappa^2 + |G^{\varepsilon,n}(s)|^2)^{\lambda - 1} \| \sigma(Y_s^\varepsilon,\delta_{X^0_s}) - \sigma(\bar{Y}_s^{\varepsilon,n},\delta_{\bar{X}^0_s}) \|_{\mathrm{HS}}^2 \, ds\nonumber\\
+&2\lambda(\lambda - 1)\varepsilon\int_0^t(\kappa^2 + |G^{\varepsilon,n}(s)|^2)^{\lambda - 2} \left| (\sigma(Y_s^\varepsilon,\delta_{X^0_s}) - \sigma(\bar{Y}_s^{\varepsilon,n},\delta_{\bar{X}^0_s}) )^* G^{\varepsilon,n}(s) \right|^2 \, ds+\Phi(t)\nonumber\\
\le &\kappa^{2\lambda}+  L\lambda(2\epsilon(2\lambda-1)+1)\int_0^t(\kappa^2 + |G^{\varepsilon,n}(s)|^2)^{\lambda - 1} \| Z_s^{\varepsilon,n} \|_\infty^2\,ds\nonumber\\
+&2L\lambda(2\lambda-1)\epsilon\int_0^t(\kappa^2 + |G^{\varepsilon,n}(s)|^2)^{\lambda - 1}(\| Y_{s}^{\varepsilon,n} -\bar{Y}_{s}^{\varepsilon,n}\|_{\infty}^2+\|X^0_s-\bar{X}^0_s\|_\infty^2)\,ds+\Phi(t),
\end{align}
where we have used assumption (A2) and $\mathbb{W}_2(\delta_{X^0_s},\delta_{\bar{X}^0_s})^2\le \|X^0_s-\bar{X}^0_s\|_\infty^2$. Moreover, $\Phi(t) := 2\lambda\sqrt{\epsilon} \int_0^t 
(\kappa^2 + |G^{\varepsilon,n}(s)|^2)^{\lambda - 1} 
\langle G^{\varepsilon,n}(s), \sigma(Y_s^{\varepsilon,n},\delta_{X^0_s}) - \sigma(\bar{Y}_s^{\varepsilon,n},\delta_{\bar{X}^0_s}) dW(s) \rangle$
is a martingale.
By the BDG inequality
\begin{align*}
&\mathbb{E} \left[ \sup_{0 \leq t \leq T} \Phi(t) \right]\\
\leq & 8\sqrt{2\varepsilon}\lambda \, \mathbb{E}  \left( \int_0^T (\kappa^2 + |G^{\varepsilon,n}(s)|^2)^{2\lambda - 2} |G^{\varepsilon,n}(s)|^2 \| \sigma(Y_s^\varepsilon,\delta_{X^0_s}) -\sigma(\bar{Y}_s^{\varepsilon,n},\delta_{\bar{X}^0_s})\|_{\mathrm{HS}}^2 \, ds \right)^{1/2}  \\
\leq& \frac{1}{2} \mathbb{E} \left[ \sup_{0 \leq t \leq T } (\kappa^2 + |G^{\varepsilon,n}(t)|^2)^{\lambda} \right] \\
+ &64 \lambda^2 \varepsilon \, \mathbb{E} \left[ \int_0^{T} (\kappa^2 + |G^{\varepsilon,n}(s)|^2)^{\lambda - 1} \| \sigma(Y_s^\varepsilon,\delta_{X^0_s}) -\sigma(\bar{Y}_s^{\varepsilon,n},\delta_{\bar{X}^0_s})\|_{\mathrm{HS}}^2 \, ds \right] \\
\leq& \frac{1}{2} \mathbb{E} \left[ \sup_{0 \leq t \leq T } (\kappa^2 + |G^{\varepsilon,n}(t)|^2)^{\lambda} \right] \\
+&128 L \lambda^2 \varepsilon \, \mathbb{E} \left[ \int_0^{T} (\kappa^2 + |G^{\varepsilon,n}(s)|^2)^{\lambda - 1} \| Z_s^{\varepsilon,n} \|_\infty^2 \, ds \right] \\
+&128 L\lambda^2 \varepsilon \, \mathbb{E} \left[ \int_0^{T} (\kappa^2 + |G^{\varepsilon,n}(s)|^2)^{\lambda - 1} \| Y_{s}^{\varepsilon,n}-\bar{Y}_{s}^{\varepsilon,n}\|_{\infty}^2 \, ds \right] \\
+& 128 L \lambda^2 \varepsilon \, \mathbb{E} \left[ \int_0^{T} (\kappa^2 + |G^{\varepsilon,n}(s)|^2)^{\lambda - 1} \| X^0_s-\bar{X}^0_s \|_\infty^2 \, ds \right]. 
\end{align*}
This ,together with (\ref{3.2}), one has
\begin{align}\label{3.8}
&\mathbb{E} \left[ \sup_{0 \leq t \leq T} (\kappa^2+|G^{\epsilon.n}(t\wedge\gamma^{\epsilon,n}(\kappa)\wedge\widetilde{\gamma}^{\epsilon,n}(\delta))|^2)^\lambda \right]\nonumber \\
\leq& 2\kappa^{2\lambda} + 2L\lambda(132\lambda\varepsilon - 2\varepsilon + 1) \int_0^{T\wedge\gamma^{\epsilon,n}(\kappa)\wedge\widetilde{\gamma}^{\epsilon,n}(\delta)} \mathbb{E} \left( \kappa^2 + |G^{\varepsilon,n}(s)|^2 \right)^{\lambda - 1} \| Z_s^{\varepsilon,n} \|_\infty^2 \, ds \nonumber\\
+& 4L\lambda\varepsilon(66\lambda - 1) \int_0^{T\wedge\gamma^{\epsilon,n}(\kappa)\wedge\widetilde{\gamma}^{\epsilon,n}(\delta)} 
\mathbb{E} \left( \kappa^2 + |G^{\varepsilon,n}(s)|^2 \right)^{\lambda - 1} \| Y_s^{\varepsilon,n} - \bar{Y}_s^{\varepsilon,n} \|_\infty^2 \, ds \nonumber\\
+& 4L\lambda\varepsilon(66\lambda - 1) \int_0^{T\wedge\gamma^{\epsilon,n}(\kappa)\wedge\widetilde{\gamma}^{\epsilon,n}(\delta)} 
\mathbb{E} \left( \kappa^2 + |G^{\varepsilon,n}(s)|^2 \right)^{\lambda - 1} \| X_s^{0} - \bar{X}_s^{0} \|_\infty^2 \, ds \nonumber\\
\leq& 2\kappa^{2\lambda} + 2L\lambda(132\lambda\varepsilon - 2\varepsilon + 1) \int_0^{T\wedge\gamma^{\epsilon,n}(\kappa)\wedge\widetilde{\gamma}^{\epsilon,n}(\delta)} \mathbb{E} \left[ \sup_{0 \leq u \leq s} 
\left( \kappa^2 + |G^{\varepsilon,n}(u)|^2 \right)^{\lambda - 1} 
\| Z_u^{\varepsilon,n} \|_\infty^2 \right] ds \nonumber\\
+& 4L\lambda\varepsilon(66\lambda - 1) \int_0^{T\wedge\gamma^{\epsilon,n}(\kappa)\wedge\widetilde{\gamma}^{\epsilon,n}(\delta)} \mathbb{E} \left[ \sup_{0 \leq u \leq s} \left( \kappa^2 + |G^{\varepsilon,n}(u)|^2 \right)^{\lambda } \right] ds \nonumber\\
+& 4L\lambda\varepsilon(66\lambda - 1) \int_0^{T\wedge\gamma^{\epsilon,n}(\kappa)\wedge\widetilde{\gamma}^{\epsilon,n}(\delta)} \mathbb{E} \left[ \sup_{0 \leq u \leq s} \left( \kappa^2 + |G^{\varepsilon,n}(u)|^2 \right)^{\lambda - 1} \| X_u^{0} - \bar{X}_u^{0} \|_\infty^2 \right] ds ,
\end{align}
where the fact that $\| Y_s^{\varepsilon,n} - \bar{Y}_s^{\varepsilon,n} \|_\infty^2\le \kappa^2$ for $s\le \gamma^{\epsilon,n}(\kappa)$ has been used by utilizing the definition of stopping time $\gamma^{\epsilon,n}(\kappa)$. By (\ref{2.3}) we have
\begin{equation}\label{3.6}
    \sup_{0 \leq u \leq s}\| Z_u^{\varepsilon,n} \|_\infty\le\frac{1}{1-\alpha}\sup_{0 \leq u \leq s}|G^{\varepsilon,n}(u)|.
\end{equation}
By similar argument as in (\ref{3.5})
\begin{align} \label{3.7}
&\sup_{0 \leq u \leq s} \| X_{u}^{0} - \bar{X}_{u}^{0} \|_{\infty}\nonumber\\ 
\leq &\frac{1}{1 - \alpha} 
\sup_{0 \leq u \leq s} \sup_{u_n - u \leq \theta \leq 0}
\left\| 
\int_{u_n}^{u+\theta} b(X_s^{0},\delta_{X^0_s}) \, ds \right\|\le\frac{L_5}{(1-\alpha)n}<\kappa.
\end{align}
where the condition $ \frac{\sqrt{d} L_5}{(1-\alpha)n}  < \kappa$ has been used. Substituting (\ref{3.6}) and (\ref{3.7}) into (\ref{3.8}), we come to 
\begin{align*}
        &\mathbb{E} \left[ \sup_{0 \leq t \leq T}  (\kappa^2+|G^{\epsilon.n}(t\wedge\gamma^{\epsilon,n}(\kappa)\wedge\widetilde{\gamma}^{\epsilon,n}(\delta))|^2)^\lambda \right]\\
    \le& 2\kappa^{2\lambda}+C_1\int_0^{T\wedge\gamma^{\epsilon,n}(\kappa)\wedge\widetilde{\gamma}^{\epsilon,n}(\delta))}\mathbb{E} \left[ \sup_{0 \leq u \leq s} (\kappa^2+|G^{\epsilon.n}(u)|^2)^\lambda \right]\,ds
\end{align*}
where $C_1=2\lambda L(132\lambda\epsilon-2\epsilon+1)\frac{1}{(1-\alpha)^2}+8\lambda L\epsilon(66\lambda-1)$. Choosing \( \lambda = \frac{1}{\varepsilon} \), 
by the Gronwall inequality, we obtain
\[
\mathbb{E} \left[ \sup_{0 \leq t \leq T} (\kappa^2+|G^{\epsilon.n}(t\wedge\gamma^{\epsilon,n}(\kappa)\wedge\widetilde{\gamma}^{\epsilon,n}(\delta))|^2)^\frac{1}{\epsilon}  \right]
\leq 2 \kappa^{2/\varepsilon} e^{C_2 T / \varepsilon},
\]
where $C_2:= 2L\left( \frac{133}{(1 - \alpha)^2} + 264 \right)$. Noting that
\[
\left( \kappa^2 + (1 - \alpha)^2 \delta^2 \right)^{1/\varepsilon} \mathbb{P}\left( \widetilde{\gamma}^{\epsilon,n}(\delta) \leq T \right)
\leq \mathbb{E} \left[ \sup_{0 \leq t \leq T} (\kappa^2+|G^{\epsilon.n}(t\wedge\gamma^{\epsilon,n}(\kappa)\wedge\widetilde{\gamma}^{\epsilon,n}(\delta))|^2)^\frac{1}{\epsilon}  \right],
\]
one has
\[
\mathbb{P} \left(\widetilde{\gamma}^{\epsilon,n}(\delta) \leq T \right)
\leq \left( \frac{2^\epsilon\kappa^2}{\kappa^2 + (1 - \alpha)^2 \delta^2} \right)^{1/\varepsilon} e^{C_2 T / \varepsilon}.
\]
This implies
\[
\limsup_{\varepsilon \to 0} \varepsilon \log \mathbb{P} \left( \widetilde{\gamma}^{\epsilon,n}(\delta) \leq T \right)
\leq \log \left( \frac{\kappa^2}{\kappa^2 + (1 - \alpha)^2 \delta^2} \right) + C_2 T.
\]
Finally, given \( L > 0 \), choose \( \kappa \) sufficiently small such that $\log \left( \frac{\kappa^2}{\kappa^2 + (1 - \alpha)^2 \delta^2} \right) + C_2 T \leq -2L$. Next, utilizing (\ref{3.9}), choose \( N \) sufficient large such that
$\limsup_{\varepsilon \to 0} \varepsilon \log \mathbb{P} \left( \gamma^{\epsilon,n}(\kappa) \leq T \right) \leq -2L
\quad \text{for all } n \geq N$. Then for \( n \geq N \), there exists \( 0 < \varepsilon_n < 1 \) such that for all \( 0 < \varepsilon \leq \varepsilon_n \),
$\mathbb{P} \left( \gamma^{\epsilon,n}(\kappa) \leq T \right) \leq e^{-L/\varepsilon}$ and $\mathbb{P} \left( \widetilde{\gamma}^{\epsilon,n}(\delta) \leq T \right) \leq e^{-L/\varepsilon}$. Therefore,  we conclude
\[
\mathbb{P} \left( \sup_{0 \leq t \leq T} |Z^{\varepsilon,n}(t)| \geq \delta \right)
\leq 2 e^{-L/\varepsilon}, \quad 0 < \varepsilon \leq \varepsilon_n.
\]
Thus,
\[
\lim_{n\to\infty}\limsup_{\varepsilon \to 0} \, \varepsilon \log 
\mathbb{P} \left( \sup_{0 \leq t \leq T} |Z^{\varepsilon,n}(t)| > \delta \right) =-\infty.
\]
\end{proof}
\begin{lemma}\label{L3.2}
    Under the same assumptions in Lemma 3.1,  the following statement holds for any $\varrho<\infty$
    \begin{equation}
        \lim_{n\to\infty}\sup_{\{\varphi;I(\varphi)\le \varrho\}}\sup_{-\tau\le t\le T}|M^n(\varphi)(t)-M(\varphi)(t)|=0.
    \end{equation}
\end{lemma}
\begin{proof}
    Denote $J^n(\varphi)(t)=M^n(\varphi)(t)-M(\varphi)(t)-(D(M_t^n(\varphi))-D(M_t(\varphi))$. Similar to (\ref{2.3}), it holds that
  \begin{equation}\label{y1}
     \sup_{0\le t\le T}|M^n(\varphi)(t)-M(\varphi)(t)|\le\frac{1}{1-\alpha}\sup_{0\le t\le T}|J^n(\varphi)(t)|.
  \end{equation}
  Applying assumption (A2), we obtain the following estimate
  \begin{align}
      &|J^n(\varphi)(t)|^2\nonumber\\ 
      \leq& 2\int_0^t |\langle J^n(\varphi)(s), b(M_s^n(\varphi),\delta_{X^0_s}) - b(M_s(\varphi),\delta_{X^0_s}) \rangle |\, ds \nonumber\\
       + &2\int_0^t |\langle J^n(\varphi)(s), (\sigma(\bar{M}_s^n(\varphi),\delta_{\bar{X}^0_s}) - \sigma(M_s(\varphi)),\delta_{X^0_s})) \dot{\varphi}(s) \rangle| \, ds\nonumber \\
      \leq& L \int_0^t \|M_s^n(\varphi) - M_s(\varphi)\|_\infty^2 \, ds 
      + \int_0^t |J^n(\varphi)(s)|^2 \, ds \nonumber\\
       +&2 L \int_0^t \|M_s^n(\varphi) - M_s(\varphi)\|_\infty^2 |\dot{\varphi}(s)|^2 \, ds 
      + 2L \int_0^t \|\bar{M}_s^n(\varphi) - M_s^n(\varphi)\|_\infty^2 |\dot{\varphi}(s)|^2 \, ds\nonumber\\
      +&2L \int_0^t \|\bar{X}_s^0 - X_s^0\|_\infty^2 |\dot{\varphi}(s)|^2 \, ds\nonumber \\
      \le& (L+(1+\alpha)^2) \int_0^t \|M_s^n(\varphi) - M_s(\varphi)\|_\infty^2 \, ds
      +2 L \int_0^t \|M_s^n(\varphi) - M_s(\varphi)\|_\infty^2 |\dot{\varphi}(s)|^2 \, ds \nonumber\\
      +& 2L \int_0^t \|\bar{M}_s^n(\varphi) - M_s^n(\varphi)\|_\infty^2 |\dot{\varphi}(s)|^2 \, ds
      +2L \int_0^t \|\bar{X}_s^0 - X_s^0\|_\infty^2 |\dot{\varphi}(s)|^2 \, ds.
  \end{align}
 Using  (\ref{y1}), one has
    \begin{align}
      &\sup_{0\le t\le T}|M^n(\varphi)(t)-M(\varphi)(t)|^2\nonumber\\
      \leq &\frac{1}{(1-\alpha)^2}\Big\{(L+(1+\alpha)^2) \int_0^T \sup_{0\le u\le s}|M^n(\varphi)(u) - M(\varphi)(u)|^2 \, ds  \nonumber\\
      +&2 L \int_0^T \sup_{0\le u\le s}|M^n(\varphi)(u) - M(\varphi)(u)|^2 |\dot{\varphi}(s)|^2 \, ds\nonumber\\ 
      +& 2L \int_0^T \sup_{0\le u\le s}\|\bar{M}_u^n(\varphi) - M_u^n(\varphi)\|_\infty^2 |\dot{\varphi}(s)|^2 \, ds
      + 2L \int_0^T \sup_{0\le u\le s}\|\bar{X}_u^0 - X_u^0\|_\infty^2 |\dot{\varphi}(s)|^2 \, ds\Big\}.
  \end{align}
  Following the same line of reasoning as in (\ref{3.5}), H\"older's inequality and condition (\ref{3.1}), we obtain
  \begin{align} 
&\sup_{0 \leq u \leq s} \|\bar{M}_u^n(\varphi) - M_u^n(\varphi)\|_\infty\nonumber\\ 
\leq &\frac{1}{1 - \alpha} 
\sup_{0 \leq u\leq s} \sup_{u_n - u \leq \theta \leq 0}
\left| \int_{u_n}^{u+\theta} b(M_r^n(\varphi),\delta_{X^0_r}) \, dr 
+  \int_{u_n}^{u+\theta} \, \sigma(\bar{M}_r^n(\varphi),\delta_{\bar{X}^0_r})\dot{\varphi}(r)\, dr\right|\nonumber\\
\leq &\frac{L_5(1+\sqrt{2\varrho})}{(1-\alpha)}\Big(\frac{1}{n}\Big)^\frac{1}{2},
\end{align}
uniformly over the set $\{\varphi;I(\varphi)\le \varrho\}$. This together with (\ref{3.7}) yields
\begin{align*}
    &\sup_{0\le t\le T}|M^n(\varphi)(t)-M(\varphi)(t)|^2\\
    \leq &\frac{1}{(1-\alpha)^2}\Big\{(L+(1+\alpha)^2) \int_0^T \sup_{0\le u\le s}|M^n(\varphi)(u) - M(\varphi)(u)|^2 \, ds  \nonumber\\
    +&2 L \int_0^T \sup_{0\le u\le s}|M^n(\varphi)(u) - M(\varphi)(u)|^2 |\dot{\varphi}(s)|^2 \, ds +\frac{4\varrho LL_5^2[(1+\sqrt{2\varrho})^2+1]}{(1-\alpha)^2n}\Big\}.
\end{align*}
By Gronwall's inequality, it follows that
$$
\sup_{0\le t\le T}|M^n(\varphi)(t)-M(\varphi)(t)|^2\le\frac{4\varrho LL_5^2[(1+\sqrt{2\varrho})^2+1]}{(1-\alpha)^4n}\exp{\Big\{\frac{T(L+(1+\alpha)^2)+4L\varrho}{(1-\alpha)^2}}\Big\}.
$$
The above estimate holds uniformly over the set \(\{\varphi ; I(\varphi) \leq \varrho\}\). Letting \(n \to \infty\), we obtain the desired result.
\end{proof}
The family of measures $\{\nu_\epsilon\}$, corresponding to the law of the scaled Brownian motion $\{\sqrt{\epsilon}W\}_\epsilon$, satisfies the LDP with the rate function $I(\varphi)$. By Lemma \ref{L3.1}, the family $\{\nu_\epsilon\circ (M^n)^{-1}\}$, i.e.,  the law of $\{Y^{\varepsilon,n}\}_\epsilon$, provides an exponentially good approximation to $\{\nu_\epsilon\circ (M)^{-1}\}$,  the law of $\{Y^{\varepsilon}\}_\epsilon$. Moreover, Lemma \ref{L3.2} ensures that $M$ can be well approximated by $M^n$. Therefore, the conditions of Lemma \ref{L2.1} are satisfied, and we can conclude the following the desired result.
\begin{lemma}\label{L3.3}
    Let (A1), (A2) and (\ref{3.1}) hold.  Then the law of $\{Y^{\varepsilon}(t)\}_{t\in[-\tau,T]}$ satisfies the LDP with the rate function
\[
\widehat{I}(f) := \inf \left\{ I(\varphi) \;\middle|\; M(\varphi) = f,\; \varphi \in H([0,T];\mathbb{R}^m) \right\}, \quad f \in C([-\tau, T]; \mathbb{R}^d).
\]
\end{lemma}
\subsection{Removal of the Boundedness Assumption}
Next, we remove the boundedness assumptions on $b$ and $\sigma$. We begin with

\begin{lemma}\label{L3.4}
    Assume that (A1)-(A3) hold, then for any $R>0$
    \begin{equation}
    \lim_{R\to\infty}\limsup_{\epsilon\to0}\epsilon\log\mathbb{P}(\sup_{-\tau\le t\le T}|Y^\epsilon(t)|>R)=-\infty.
    \end{equation}
\end{lemma}
\begin{proof}
    Define the stopping time $\zeta^{\epsilon,R}=\inf\{t\ge0:|Y^\epsilon(t)|>R\}$, let $G^\epsilon(t)=Y^\epsilon(t)-D(Y^\epsilon_t)$ and $G^{\epsilon,\xi}(t)=G^{\epsilon}(t\wedge\zeta^{\epsilon,R})$. For any $\lambda>0$, applying It\^o's formula to $(1+|G^{\epsilon}(t)|^2)^\lambda$ yields
    \begin{align} \label{3.12}  
&(1+|G^{\epsilon}(t)|^2)^\lambda\nonumber\\
=& (1+|G^{\epsilon}(0)|^2)^\lambda
+2\lambda \int_0^{t} (\rho^2 + |G^{\varepsilon}(s)|^2)^{\lambda - 1} 
\left\langle G^{\varepsilon}(s), b(Y_s^\varepsilon,\delta_{X^0_s})  \right\rangle \, ds\nonumber\\
+&\lambda \varepsilon \int_0^{t}(\rho^2 + |G^{\varepsilon}(s)|^2)^{\lambda - 1} \| \sigma(Y_s^\varepsilon,\delta_{X^0_s}) \|_{\mathrm{HS}}^2\,ds\nonumber\\
+&2\lambda(\lambda - 1)\varepsilon\int_0^{t}(\rho^2 + |G^{\varepsilon}(s)|^2)^{\lambda - 2} \left| \sigma(Y_s^\varepsilon,\delta_{X^0_s})  G^{\varepsilon,n}(s) \right|^2\,ds+\tilde{\Phi}(t)\nonumber\\
\le&\lambda L_2(2\lambda\epsilon-\epsilon + 1)\int_0^{t}(1 + |G^{\varepsilon}(s)|^2)^{\lambda - 1} (1+\|Y^\epsilon_s\|_\infty^2+\|X^0_s\|_\infty^2)\,ds+\tilde {\Phi}(t),
\end{align}
where $\tilde{\Phi}(t) := 2\lambda\sqrt{\epsilon} \int_0^{t} 
(1+ |G^{\varepsilon}(s)|^2)^{\lambda - 1} 
\langle G^{\varepsilon}(s), \sigma(Y_s^{\varepsilon},\delta_{X^0_s}) dW(s) \rangle$ is a martingale, and assumption (A2) has been used. By the BDG inequality
\begin{align}\label{3.11}
&\mathbb{E} \left[ \sup_{0 \leq t \leq T} \tilde{\Phi}(t) \right]\nonumber\\
\leq & 8\sqrt{2\varepsilon}\lambda \, \mathbb{E}  \left( \int_0^T  (1 + |G^{\varepsilon}(s)|^2)^{2\lambda - 2} |G^{\varepsilon}(s)|^2 \| \sigma(Y_s^\varepsilon,\delta_{X^0_s})\|_{\mathrm{HS}}^2 \, ds \right)^{1/2}  \nonumber\\
\leq& \frac{1}{2} \mathbb{E} \left[ \sup_{0 \leq t \leq T } (1 + |G^{\varepsilon}(t)|^2)^{\lambda} \right] \nonumber\\
+&64 L_2 \lambda^2 \varepsilon \, \mathbb{E} \left[ \int_0^{T} (1 + |G^{\varepsilon}(s)|^2)^{\lambda - 1} (1+\|Y^\epsilon_s\|_\infty^2+\| X_s^0 \|_\infty^2 )\, ds \right] .
\end{align}
Combining (\ref{3.12}) and (\ref{3.11}) yields
\begin{align}
&\mathbb{E} \left[ \sup_{0 \leq t \leq T} (1+|G^{\epsilon}(t\wedge\zeta^{\epsilon,R})|^2)^\lambda \right]\nonumber \\
\leq& 2(1+(1+\alpha)^2\|\zeta\|_\infty^2)^\lambda \nonumber\\
+& 2L_2\lambda(66\lambda\varepsilon - \varepsilon + 1) \mathbb{E}\int_0^{T\wedge\zeta^{\epsilon,R}}  \sup_{0\le u\le s}\left( 1+ |G^{\varepsilon}(u)|^2 \right)^{\lambda - 1} (1+\|Y^{\epsilon}_u\|_\infty^2+\| X_u^{0} \|_\infty^2) \, ds .
\end{align}
Similar to (\ref{1.4}), one can get
\begin{equation}
    \sup_{0\le u \le s}\|Y^{\epsilon}_u\|_\infty^2\le\frac{1}{1-\alpha}\|\zeta\|_\infty^2+\frac{1}{(1-\alpha)^2}\sup_{0\le u\le s}|G^\epsilon(u)|^2.
\end{equation}
This, together with (\ref{2.1}), implies
\begin{align*}
    &\mathbb{E} \left[ \sup_{0 \leq t \leq T} (1+|G^{\epsilon,\zeta}(t)|^2)^\lambda \right]\nonumber \\
\leq& 2(1+(1+\alpha)^2\|\zeta\|_\infty^2)^\lambda \nonumber\\
+& 2L_2C_3\lambda(66\lambda\varepsilon - \varepsilon + 1) \mathbb{E}\int_0^T  \sup_{0\le u\le s}\left( 1+ |G^{\varepsilon,\zeta}(u)|^2 \right)^{\lambda }  \, ds ,
\end{align*}
where $C_3=1+\frac{1}{1-\alpha}\|\zeta\|_\infty^2+\frac{(\alpha^2+\alpha+1)\|\zeta\|_\infty^2+L_2T}{(1-\alpha)^2}e^{\frac{2L_2T}{(1-\alpha)^2}}$. By the Gronwall inequality, it follows that
\[
\mathbb{E} \left[ \sup_{0 \leq t \leq T} (1+|G^{\epsilon}(t\wedge\zeta^{\epsilon,R})|^2)^\lambda \right]\le2(1+(1+\alpha)^2\|\zeta\|_\infty^2)^\lambda\exp{\{2L_2C_3\lambda(66\lambda\varepsilon - \varepsilon + 1)T\}}.
\]
Noting that
\[
(1+(1-\alpha)^2R-(1-\alpha)\|\zeta\|_\infty^2)^\lambda\mathbb{P}(\zeta^{\epsilon,R}\le T)\le\mathbb{E} \left[ \sup_{0 \leq t \leq T} (1+|G^{\epsilon,\zeta}(t)|^2)^\lambda \right],
\]
we have
\begin{align*}
  & \mathbb{P}(\sup_{-\tau\le t\le T}|Y^\epsilon(t)|>R)\le\mathbb{P}(\zeta^{\epsilon,R}\le T)\\
  \le&\frac{2(1+(1+\alpha)^2\|\zeta\|_\infty^2)^\lambda\exp{\{2L_2C_3\lambda(66\lambda\varepsilon - \varepsilon + 1)T\}}}{(1+(1-\alpha)^2R-(1-\alpha)\|\zeta\|_\infty^2)^\lambda}. 
\end{align*}
Choosing $\lambda=\frac{1}{\epsilon}$ yields
    \begin{equation*}
    \lim_{R\to\infty}\limsup_{\epsilon\to0}\epsilon\log\mathbb{P}(\sup_{-\tau\le t\le T}|Y^\epsilon(t)|>R)=-\infty.
    \end{equation*}
    The proof is therefore complete.
\end{proof}
For any \( R > 0 \), define the truncated coefficients
$$
b_R(\xi, \delta_{X^0_t}) = \chi_R(\xi)\, b(\xi, \delta_{X^0_t}),\quad \sigma_R(\xi, \delta_{X^0_t}) = \chi_R(\xi)\, \sigma(\xi, \delta_{X^0_t})
$$
where
$$
\chi_R(\xi) =
\begin{cases}
1, & \|\xi\|_\infty  \le R, \\[6pt]
R+1 - \|\xi\|_\infty , & R <\|\xi\|_\infty  < R+1, \\[6pt]
0, & \|\xi\|_\infty  \ge R+1.
\end{cases}
$$


\begin{remark}
    It is clear that both $b_R$ and $\sigma_R$ satisfy assumption (A3). We now proceed to show that they also satisfy assumption (A2). 
    \begin{align*}
        & 2 \langle \xi(0) - \eta(0) - (D(\xi) - D(\eta)), b_R(\xi,  \delta_{X^0_t}) - b_R(\eta,  \delta_{X^0_t}) \rangle\\
        =& \, 2 \chi_R(\xi) \langle \xi(0) - \eta(0) - (D(\xi) - D(\eta)), b(\xi,  \delta_{X^0_t}) - b(\eta,  \delta_{X^0_t}) \rangle\\
        +&2\,(\chi_R(\xi)-\chi_R(\eta)) \langle \xi(0)-\eta(0) - ( D(\xi)-D(\eta)),b(\eta,  \delta_{X^0_t}) \rangle\\
        \le& L\|\xi- \eta\|_\infty^2\\
        +&2\,(\chi_R(\xi)-\chi_R(\eta)) \langle \xi(0)-\eta(0) - ( D(\xi)-D(\eta)),b(\eta,  \delta_{X^0_t}) \rangle
     \end{align*}
    Note that $|\chi_R(\xi)-\chi_R(\eta)|\le\|\xi-\eta\|_\infty$. For the case that $\|\xi\|_\infty\vee\|\eta\|_\infty\ge R+1$, $\chi_R(\xi)-\chi_R(\eta)=0$, the second term of the above inequality vanishes. Hence it is sufficient to assume that $\|\xi\|_\infty<R+1$ or $\|\eta\|_\infty<R+1$. This together with (\ref{2.1}) indicates that $(\eta,  \delta_{X^0_t})$ is bounded in $\mathscr{C} \times \mathscr{P}_2(\mathscr{C})$. Hence, based on (A2), there exists a positive constant $L_R$ such that $|b(\eta,  \delta_{X^0_t})|\vee\||\sigma(\eta,  \delta_{X^0_t})\|_{HS}\le L_R$. These imply 
        \begin{align*}
        & 2 \langle \xi(0) - \eta(0) - (D(\xi) - D(\eta)), b_R(\xi,  \delta_{X^0_t}) - b_R(\eta,  \delta_{X^0_t}) \rangle\\
        \le & (L+2L_R(1+\alpha))\| \xi-\eta\|_\infty^2.  
        \end{align*} 
 By a similar argument, one can show that $\sigma_R$ also satisfies assumption (A2).
    \begin{align*}
         & \|\sigma_R(\xi, \delta_{X^0_t}) - \sigma_R(\eta, \delta_{X^0_t})\|_{\mathrm{HS}}^2 \\
         \le &2\chi_R(\xi)^2\|\sigma(\xi, \delta_{X^0_t}) - \sigma(\eta, \delta_{X^0_t})\|_{\mathrm{HS}}^2+2|\chi_R(\xi)-\chi_R(\eta)|^2\|\sigma_R(\sigma(\eta, \delta_{X^0_t})\|_{\mathrm{HS}}^2 \\
         \le&2(L+L_R^2)\|\xi- \eta\|_\infty^2.
    \end{align*}
\end{remark}
Consider   the solution $Y^{\epsilon, R}(t)$ of the following SDE:
\begin{equation}
        d(Y^{\epsilon,R}(t)-D(Y^{\epsilon,R}_t))=b_R(Y^{\epsilon,R}_t, \delta_{X^0_t})dt+\sqrt{\epsilon}\sigma_R(Y^{\epsilon,R}_t, \delta_{X^0_t})dW(t),\quad t\in[0,T],
\end{equation}
with $Y^{\epsilon,R}_0=\xi$. For \( \varphi \) with \( I(\varphi) < \infty \), let \( M^R(\varphi) \) be the solution to the following equation:
\begin{align*}
M^R(\varphi)(t) - D(M_{t}^{R}(\varphi)) &= M^R(\varphi)(0) - D(M_{0}^{R}(\varphi)) + \int_0^t b_R(M_{s}^{R}(\varphi),\delta_{X^0_s})\,ds \\
&+ \int_0^t \sigma_R(M_{s}^{R}(\varphi),\delta_{X^0_s})\,\dot{\varphi}(s)\,ds,
\end{align*}
with the initial condition $M_0^{R}(\varphi) = \xi$. Define the rate function
\[
I_R(f) := \inf \left\{ I(\varphi)\;\middle|\; M^R(\varphi) = f \right\}, \quad f \in C([-\tau, T]; \mathbb{R}^d).
\]
If $\sup_{-\tau \le t \le T} |M(\varphi)(t)| \le R$, then \( M(\varphi) = M^R(\varphi) \), and it follows that
\[
\widehat{I}(f) = I_R(f), \quad \text{for all } f \text{ such that } \sup_{-\tau \le t \le T} |f(t)| \le R.
\]
Applying the Lemma \ref{L3.3}, we conclude that the law of $Y^{\varepsilon,R}(\cdot)$, denoted by $\mu^{\epsilon,R}$, satisfies the LDP with the rate function $I_R(f)$. The following lemma establishes that the processes $Y^{\varepsilon,R}(t)$ are exponentially good approximation of $Y^{\varepsilon}(t)$.
\begin{lemma}
    Suppose that Assumptions (A1)-(A3) are satisfied. Then for any $\delta>0$, the following holds
    \begin{equation}
        \lim_{R\to\infty}\limsup_{\epsilon\to0}\epsilon\log\mathbb{P}(\sup_{-\tau\le t\le T}|Y^\epsilon(t)-Y^{\epsilon,R}(t)
        |>\delta)=-\infty  
\end{equation}
\end{lemma}
\begin{proof}
    Let $Z^{\epsilon,R}(t)=Y^{\epsilon}(t)-Y^{\epsilon,R}(t)$, and define the auxiliary process $G^{\epsilon.R}(t)=Y^{\epsilon}(t)-Y^{\epsilon,R}(t)-(D(Y^\epsilon_t)-D(Y^{\epsilon,R}_t))$. For any $R>0$, define the stopping times $\zeta^{\epsilon}(R)=\inf\{t\ge0: |Y^{\epsilon}(t)|\ge R\}$, and $\widehat{\zeta}^{\epsilon}(\delta)=\inf\{t\ge0: |Z^{\epsilon,R}(t\wedge\zeta^{\epsilon}(R))|\ge \delta\}$. Then we can estimate
        \begin{align}\label{3.15}
        &\mathbb{P}(\sup_{0\le t\le T}|Z^{\epsilon,R}(t)
        |>\delta)\nonumber\\
        =&\mathbb{P}(\sup_{0\le t\le T}|Z^{\epsilon,R}(t)
        |>\delta, \zeta^{\epsilon}(R)\le T)+\mathbb{P}(\sup_{0\le t\le T}|Z^{\epsilon,R}(t)
        |>\delta, \zeta^{\epsilon}(R)> T)\nonumber\\
        \le &\mathbb{P}(\zeta^{\epsilon}(R)\le T)+\mathbb{P}(\widehat{\zeta}^{\epsilon}(\delta)\le T)\nonumber\\
        \le &\mathbb{P}(\sup_{0\le t \le T}|Y^\epsilon(t)|\ge R)+\mathbb{P}(\widehat{\zeta}^{\epsilon}(\delta)\le T).
    \end{align}
   To estimate the second term, we apply an argument similar to that used in Lemma \ref{L3.1}, which implies that for any $\rho>0$ 
    \[
    \mathbb{E} \left[ \sup_{0 \leq t \leq T} (\rho^2+|G^{\epsilon,R}(t\wedge\zeta^{\epsilon}(R)\wedge\hat{\zeta}^{\epsilon}(\delta))|^2)^\frac{1}{\epsilon} \right]
   \leq 2 \rho^{2/\varepsilon} e^{C T / \varepsilon},
    \]
for some constant \( C > 0 \) independent of $\epsilon$. Hence,
\[
    \mathbb{P} \left( \widehat{\zeta}^{\epsilon}(\delta) \leq T \right)
   \leq \left( \frac{2^\epsilon\rho^2}{\rho^2 + (1 - \alpha)^2 \delta^2} \right)^{1/\varepsilon} e^{CT / \varepsilon}.
\]
Taking logarithm and upper limit as \( \varepsilon \to 0 \), we obtain
\begin{equation}\label{3.16}
 \limsup_{\varepsilon \to 0} \varepsilon \log \mathbb{P} \left( \widehat{\zeta}^{\epsilon}(\delta) \leq T \right)
\leq \log \left( \frac{\rho^2}{\rho^2 + (1 - \alpha)^2 \delta^2} \right) + CT.   
\end{equation}
Combining estimates (\ref{3.15}) and (\ref{3.16}) yields that
\begin{align}
    &\lim_{R\to\infty}\limsup_{\epsilon\to0}\epsilon\log\mathbb{P}(\sup_{-\tau\le t\le T}|Z^{\epsilon,R}(t)|>\delta)\nonumber\\
    \le&\big(\limsup_{\epsilon\to0}\epsilon\log\mathbb{P}(\sup_{0\le t \le T}|Y^\epsilon(t)|\ge R)\big)\vee \Big\{\log \left( \frac{\rho^2}{\rho^2 + (1 - \alpha)^2 \delta^2} \right) + CT\Big\},
\end{align}
Finally, letting $\rho\to0$ and then $R\to\infty$, and using Lemma \ref{L3.4}, we obtain the desired conclusion.
\end{proof}

\begin{proof}[Proof of (i) of Theorem 2.1.] 
   It is sufficient to prove the upper bound condition (\ref{3.17}) and the lower bound condition (\ref{3.16}) stated below according to the definition of LDP.
    \begin{equation}\label{3.17}
    \limsup_{\varepsilon \to 0} \varepsilon \log \mu_\varepsilon(C) \le - \inf_{f \in C} \widehat{I}(f),\quad \forall \quad\text{closed subset} \quad C \subset C([-\tau, T]; \mathbb{R}^d),
    \end{equation}
    \begin{equation}\label{3.18}
    \liminf_{\varepsilon \to 0} \varepsilon \log \mu_\varepsilon(G) \ge - \inf_{f \in G}\widehat{I}(f),\quad \forall \quad\text{open subset} \quad G \subset C([-\tau, T]; \mathbb{R}^d),
    \end{equation}
   where $\mu_\epsilon$ is the law of $Y^\epsilon(\cdot)$ on $C([-\tau, T]; \mathbb{R}^d)$. For \( R > 0 \) and a closed subset \( C \subset C([-\tau, T]; \mathbb{R}^d) \), define $C_R := C \cap \left\{ f:\|f\|_\infty \le R \right\}$. Let \( C_R^\delta \) denote the closed \( \delta \)-neighborhood of \( C_R \). Then we have the following estimate
\begin{align*}
\mu^\varepsilon(C) 
&\le \mu^\varepsilon(C_{R_1}) + \mathbb{P} \left( \sup_{-\tau \le t \le T} |Y^\varepsilon(t)| > R_1 \right) \\
&\le \mathbb{P} \left( \sup_{-\tau \le t \le T} |Y^\varepsilon(t) - Y^{\varepsilon, R}(t)| > \delta \right)
+ \mu^{\varepsilon, R} \left( C_{R_1}^\delta \right) 
+ \mathbb{P} \left( \sup_{-\tau \le t \le T} |Y^\varepsilon(t)| > R_1 \right), 
\end{align*}
where $\mu^{\epsilon,R}$ is the law of $Y^{\epsilon,R}(\cdot)$ on $C([-\tau, T]$. Taking the LDP of $\mu^{\epsilon,R}$ into consideration, we obtain
\begin{align*}
\limsup_{\varepsilon \to 0} \varepsilon \log \mu^\varepsilon(C)
&\le \left( - \inf_{f \in C_{R_1}^\delta} I_R(f) \right)
\vee \left( \limsup_{\varepsilon \to 0} \varepsilon \log \mathbb{P} \left( \sup_{-\tau \le t \le T} |Y^\varepsilon(t)| > R_1 \right) \right) \\
&\quad \vee \left( \limsup_{\varepsilon \to 0} \varepsilon \log \mathbb{P} \left( \sup_{-\tau \le t \le T} |Y^\varepsilon(t) - Y^{\varepsilon, R}(t)| > \delta \right) \right).
\end{align*}
By letting $R\to\infty$, $\delta\to\infty$ and $R_1\to\infty$ successively in the inequality above, we obtain the upper bound (\ref{3.17}).
Let \( G \) be an open subset of \( C([-\tau, T]; \mathbb{R}^d) \). For any \(\varphi_0 \in G\) and \(\delta > 0\), define $B(\varphi_0, \delta) = \{ f : \|f - \varphi_0\|_\infty \leq \delta \} \subset G.$
By applying the LDP for $\mu^{\varepsilon,R}$, we have
\begin{align*}
- I_R(\varphi_0) &\leq \liminf_{\varepsilon \to 0} \varepsilon \log \mu^{\varepsilon,R} \big( B(\varphi_0, \frac{\delta}{2}) \big)\\
&= \liminf_{\varepsilon \to 0} \varepsilon \log
\mathbb{P}\Big(\sup_{-\tau \le t \le T} |Y^{\varepsilon,R}(t) - \varphi_0| \le\frac{\delta}{2} \Big) \\
&\leq \liminf_{\varepsilon \to 0} \varepsilon \log \mu^\varepsilon(G) \ \vee \ \liminf_{\varepsilon \to 0} \varepsilon \log \mathbb{P} \left( \sup_{-\tau \le t \le T} |Y^\varepsilon(t) - Y^{\varepsilon,R}(t)| > \delta/2 \right)
\end{align*}
Note that \( I_R(\varphi_0) = \widehat{I}(\varphi_0) \) provided that \(\|\varphi_0\|_\infty \leq R\). Hence, letting \( R \to \infty \), we obtain
\[
- \widehat{I}(\varphi_0) \leq \liminf_{\varepsilon \to 0} \varepsilon \log \mu^\varepsilon(G).
\]
Since \(\varphi_0 \in G\) is arbitrary, it follows that
\[
- \inf_{f \in G} \widehat{I}(f) \leq \liminf_{\varepsilon \to 0} \varepsilon \log \mu^\varepsilon(G),
\]
which establishes the lower bound (\ref{3.18}) in Theorem 3.1. The proof is therefore complete.
\end{proof}

\section{Proof of (ii) of Theorem 2.1 }
Finally, to complete the argument, we establish the exponential equivalence between \(Y^\varepsilon(t)\)and \(X^\varepsilon(t)\), which leads to the central result, namely the validity of Theorem \ref{T2.1}(ii).

\begin{proof}[Proof of (ii) of Theorem 2.1.] It is sufficient to prove that the processes $X^{\varepsilon}(t)$ and $Y^{\varepsilon}(t)$ are exponentially equivalent, i.e. for any $\delta>0$, the following holds
    \begin{equation}
        \limsup_{\epsilon\to0}\epsilon\log\mathbb{P}(\sup_{-\tau\le t\le T}|X^\epsilon(t)-Y^{\epsilon}(t)
        |>\delta)=-\infty .
\end{equation}
Let us denote the difference between the two processes by $F^{\epsilon}(t)=X^{\epsilon}(t)-Y^{\epsilon}(t)$, and define the auxiliary process $H^{\epsilon}(t)=X^{\epsilon}(t)-Y^{\epsilon}(t)-(D(X^\epsilon_t)-D(Y^{\epsilon}_t))$. For any $M>0$, define the stopping time $\gamma^\epsilon(M)=\inf\{t\ge0: |Y^{\epsilon}(t)|\ge M\}$. Additionally, define another stopping time $\widetilde{\gamma}^\epsilon(\delta)=\inf\{t\ge0:|F^{\epsilon}(t\wedge \gamma^\epsilon(M))|\ge \delta\}$. By an argument similar to that used in (\ref{3.15})
    \begin{align}\label{3.20}
        &\mathbb{P}(\sup_{0\le t\le T}|F^{\epsilon}(t)
        |>\delta)\nonumber\\
        =&\mathbb{P}(\sup_{0\le t\le T}|F^{\epsilon}(t)
        |>\delta, \gamma^\epsilon(M)\le T)+\mathbb{P}(\sup_{0\le t\le T}|F^{\epsilon}(t)
        |>\delta, \gamma^\epsilon(M)> T)\nonumber\\
        \le &\mathbb{P}(\gamma^\epsilon(M)\le T)+\mathbb{P}(\widetilde{\gamma}^\epsilon(\delta)\le T)\nonumber\\
        \le &\mathbb{P}(\sup_{0\le t \le T}|Y^\epsilon(t)|\ge M)+\mathbb{P}(\widetilde{\gamma}^\epsilon(\delta)\le T).
    \end{align}
    As in the proof of Lemma \ref{L3.1} , by Itô's formula, for any $\rho>0$ and $\lambda>0$
    \begin{align}   
&(\rho^2+|H^{\epsilon}(t)|^2)^\lambda \nonumber\\
\le& \rho^{2\lambda}
+ L\lambda(2\lambda\epsilon-\epsilon+2)\int_0^{t}  (\rho^2 + |H^{\varepsilon}(s)|^2)^{\lambda - 1} (\| F_s^{\varepsilon} \|_\infty^2 
   +\mathbb{E}\|X^\epsilon_s-X^0_s\|_\infty^2) \, ds+ \widetilde{\Phi}(t),
\end{align}
where $\widetilde{\Phi}(t) := 2\lambda\sqrt{\epsilon} \int_0^{t} 
(\rho^2 + |H^{\varepsilon}(s)|^2)^{\lambda - 1} 
\langle H^{\varepsilon}(s), \sigma(X_s^{\varepsilon},\mathscr{L}_{X^\epsilon_s}) - \sigma(Y_s^{\varepsilon},\delta_{X^0_s}) dW(s) \rangle$
is a martingale. Hence, applying the BDG inequality yields
\begin{align*}
&\mathbb{E} \left[ \sup_{0 \leq t \leq T} \widetilde{\Phi}(t)\right]\\
\leq& \frac{1}{2} \mathbb{E} \left[ \sup_{0 \leq t \leq T} (\rho^2 + |H^{\varepsilon}(t)|^2)^{\lambda} \right] \\
+&64 L \lambda^2 \varepsilon \, \mathbb{E} \left[ \int_0^{T} (\rho^2 + |H^{\varepsilon}(s)|^2)^{\lambda - 1} (\| F_s^{\varepsilon} \|_\infty^2+\mathbb{E}\| X^\epsilon_s-X^0_s \|_\infty^2) \, ds \right].
\end{align*}
This, together with (\ref{2.2}), yields
\begin{align}
    &\mathbb{E} \left[ \sup_{0 \leq t \leq T} (\rho^2+|H^{\epsilon}(t\wedge\gamma^\epsilon(M)\wedge\widetilde{\gamma}^\epsilon(\delta))|^2)^\lambda  \right]\nonumber\\
\le&2\rho^{2\lambda}+C_4\int_0^{T\wedge\gamma^\epsilon(M)\wedge\widetilde{\gamma}^\epsilon(\delta)}\mathbb{E} \left[ \sup_{0 \leq u \leq s} (\rho^2+|H^{\epsilon}(u)|^2)^\lambda \right]\,ds,
\end{align}
where $C_4=\frac{2L\lambda(66\lambda\epsilon-\epsilon+2)L_4(\epsilon)}{\rho^2(1-\alpha)^2}$. Taking $\lambda=\frac{1}{\epsilon}$ and $\rho=\sqrt{\epsilon}$, 
by the Gronwall inequality, we obtain
\[
\mathbb{E} \left[ \sup_{0 \leq t \leq T} (\epsilon+|H^{\epsilon}(t\wedge\gamma^\epsilon(M)\wedge\widetilde{\gamma}^\epsilon(\delta))|^2)^\lambda  \right]
\leq 2 \epsilon^{1/\varepsilon} e^{C_5\varepsilon},
\]
where $C_5:=  \frac{17680LL_2T^2(1+2L_3(1))}{(1 - \alpha)^2}\exp\{\frac{4L}{(1-\alpha)^2}\} $. Hence
\[
\limsup_{\varepsilon \to 0} \varepsilon \log \mathbb{P} \left( \widetilde{\gamma}^\epsilon(\delta) \leq T \right)
\leq\limsup_{\varepsilon \to 0} \log \left( \frac{2^\epsilon\epsilon}{\epsilon + (1 - \alpha)^2 \delta^2} \right) + C_5.
\]
Substituting this into (\ref{3.20}), we have
\begin{align}
    &\limsup_{\epsilon\to0}\epsilon\log\mathbb{P}(\sup_{-\tau\le t\le T}|F^{\epsilon}(t)|>\delta)\nonumber\\
    \le&\big(\limsup_{\epsilon\to0}\epsilon\log\mathbb{P}(\sup_{0\le t \le T}|Y^\epsilon(t)|\ge M)\big)\vee \Big\{\limsup_{\epsilon\to0}\log \left( \frac{2^\epsilon\epsilon}{\epsilon + (1 - \alpha)^2 \delta^2} \right) + C_5\Big\},
\end{align}
Finally, the required assertion follows by letting $M\to\infty$ and using Lemma \ref{L3.4}.
\end{proof}

\section{Acknowledgements}
The first author was fully supported by a studentship from the Heilbronn Institute for Mathematical Research, UKRI Grant EP/V521917/1.

\end{document}